\documentclass[12pt,twoside]{amsart}
\usepackage{amssymb}
\usepackage{a4wide}
\usepackage[latin1]{inputenc}
\usepackage[T1]{fontenc}
\usepackage{times}

\def\hpq0{h^{p,q}_{\leq 0}}
\def\Hpq0{\H_{\leq 0}^{p,q}}

\def\dbar{\bar\partial}

\def\R{{\mathbb R}}

\def\C{{\mathbb C}}

\def\F{{\mathcal F}}

\def\H{{\mathcal H}}

\def\dhash{{d^{\#}}}

\def\n_jhash{{n_j^{\#}}}

\def\be{\begin{equation}}
\def\ee{\end{equation}}

\newtheorem{thm}{Theorem}[section]
\newtheorem{lma}[thm]{Lemma}
\newtheorem{cor}[thm]{Corollary}
\newtheorem{prop}[thm]{Proposition}

\theoremstyle{definition}

\theoremstyle{remark}

\newtheorem{preremark}{Remark}
\newtheorem{preex}{Example}

\numberwithin{equation}{section}

\title[]
{Superforms, supercurrents, minimal manifolds and Riemannian geometry.}

\address{Department of Mathematics\\Chalmers University
  of Technology \\
 S-412 96
  G\"OTEBORG\\SWEDEN} 

\email{ bob@chalmers.se}

\author[]{Bo Berndtsson}

\begin{document}

\begin{abstract}
Supercurrents, as introduced by Lagerberg, were mainly motivated as a way to study tropical varieties. Here we will associate a supercurrent to any smooth submanifold of $\R^n$. Positive supercurrents resemble positive currents in complex analysis, but  depend on a choice of scalar product on $\R^n$ and reflect the induced Riemannian structure on the submanifold. In this way we can use techniques from complex analysis to study real submanifolds. We illustrate the idea by giving area estimates of minimal manifolds and a monotonicity property of the mean curvature flow. We also illustrate the idea by a relatively short proof of Weyl's tube formula.
\end{abstract}

\maketitle

\section{Introduction}

A {\it superform} on $\R^n$ is defined as a differential form on $\C^n$ whose coefficients do not depend on the imaginary part of the variable. The dual of the space of superforms (with coefficients compactly supported in $\R^n$ and with the usual topology from the theory of distributions), is the space of {\it supercurrents}. Superforms and supercurrents were introduced by Lagerberg ,\cite{Lagerberg}, as a way to study tropical varities. A tropical variety in $\R^n$ defines a $d$-closed, positive, supercurrent of integration, and conversely any such supercurrent defines a tropical variety, given a condition on the dimension of the support. (See  also the work of Babaee, \cite{Babaee}, for related work using standard currents on $(\C^*)^n$ instead of supercurrents.)

Here we extend the `superformalism' in a different direction by associating to any smooth (or piecewise smooth) submanifold, $M$,  of $\R^n$ a supercurrent, $[M]_s$,   with the aim to apply methods from complex analysis to real manifolds.  These supercurrents are $d$-closed only if the manifold is a linear subspace, but $d[M]_s$ is given by an explicit formula involving the second fundamental form of $M$. As a result, it turns out that $M$ is minimal if and only if 
$$
[M]_s\wedge\beta^{m-1}/(m-1)!
$$
is closed, where $m$ is the dimension of $M$ and $\beta$ is the Euclidean K\"ahler form on $\C^n$ (Corollary 5.2).    Thus, minimality is characterized by a rather simple linear equation, which suggests a generalization of minimal manifolds to minimal `supercurrents'. This is of course similar to the use of (classical) currents and varifolds in the theory of minimal manifolds, but has the extra feature of a bidegree, as in complex analysis. With this we can imitate  Lelong's method for positive closed currents to prove e. g. the monotonicity formula for minimal manifolds, and a volume estimate that generalizes a recent result of Brendle and Hung, \cite{Brendle-Hung}(Theorem 6.3). We also obtain a result on removable singularities for minimal manifolds along the lines of the the El Mir-Skoda theorem from complex analysis (Theorem 7.2), and a formula for the variation of the volume under the mean curvature flow (Theorem 8.1). ( I take the opportunity to thank Duong Phong for suggesting to apply the formalism to the mean curvature flow.)

After that we give an expression of the Riemann curvature tensor of $M$ as a superform. This is basically a rewrite of Gau\ss's formula. We apply it  in the last section,  to give a rather short proof of Weyl's tube theorem, \cite{Weyl}. The proof is in essence the same as Weyl's proof, but we have included it, hoping to show that the superformalism is useful in computations.

\section{Preliminaries.}

We start by recalling the  definitions and basic properties of superforms and supercurrents, mainly following Lagerberg, \cite{Lagerberg} (and \cite{Berndtsson}), but with some modifications. Let $E$ be an $n$-dimensional vector space over $\R$. Thus $E$ can be identified with $\R^n$, but at some points it will be convenient not to fix a basis. We define the 'superspace' of $E$ to be
$$
E_s=E\oplus E= E_0\oplus E_1,
$$
where we use the subscripts to indicate the first or second summand. A superform on $E$ is a differential form on $E_s$ that is invariant under translation in the $E_1$-variable. If $x=(x_1, ...x_n)$ and $\xi=(\xi_1, ...\xi_n)$ are coordinates on $E_0$ and $E_1$ respectively a superform can then be written
\be
a=\sum a_{I, J}(x) dx_I\wedge d\xi_J,
\ee
where the coefficients do not depend on $\xi$. We say that a superform $a$ has bidegree $(p,q)$ if the lengths of the multiindices in (2.1) satisfy
$|I|=p$ and $|J|=q$. With these conventions, a superform of bidegree $(0,0)$ can be identified with a function  on $E$, since it does not depend on $\xi$. So, a `superfunction' on $E_s$ is a function on $E$.

We also equip $T^*(E_s)=T^*(E_0)\oplus T^*(E_1)$ with a complex structure $J$, such that $J$ maps $T^*(E_0)$ to $T^*(E_1)$ and vice versa. (Here our definitions differ from Lagerberg's, who considers instead maps from $T^*(E_0)$ to $T^*(E_1)$ satisfying $J^2=1$.) In the sequel we will only consider bases of $E_0$ and $E_1$ such that $J(dx_i)=d\xi_i$, and therefore $J(d\xi_i)=-dx_i$. We then extend $J$ to act on forms of arbitrary bidegree so that
$J(a\wedge b)=J(a)
\wedge J(b)$. When  $a$ is of bidegree $(p,0)$ we sometimes write $J(a)=a^{\#}$. 

The ordinary  exterior derivative of $a$ is
$$
da=\sum da_{I,J}\wedge dx_I\wedge d\xi_J= \sum \frac{\partial a_{I,J}}{\partial x_k} dx_k\wedge dx_I\wedge d\xi_J,
$$
and we also define 
$$
\dhash a:=\sum \dhash a_{I,J}\wedge dx_I\wedge d\xi_J=\sum \frac{\partial a_{I,J}}{\partial x_k} d\xi_k\wedge dx_I\wedge d\xi_J.
$$
$E_s$ is thus just the complexification of $E$ and $\dhash$ is just $d^c= i(\dbar-\partial)$, but we write $\dhash$ to emphasize that it acts only on superforms, i. e. forms not depending on $\xi$. Note that we also have
$$
\dhash a= (-1)^kJdJ(a),
$$
where $k=p+q$ is the total degree of $a$.

We   also suppose given a scalar product on $E$ and extend it to $E_s$ so that it is invariant under $J$ and $E_0\perp E_1$. Thus, if $dx_i$ are orthonormal, $d\xi_i=dx_i^{\#}$ are orthonormal on $E_1$ and $dx_i, d\xi_i$ are orthonormal on $E_s$.

The main point in the construction of $E_s$ is the definition of integrals. If $a$ is a form of maximal bidegree $(n,n)$, we write $a=a_0 dx\wedge d\xi$ with $dx=dx_1\wedge ...dx_n, d\xi=dx^{\#}$ and put
$$
\int_{E_s} a=\int_{E_0} a_0 dx\int_{E_1} d\xi.
$$
The integral over $E_0$ here is  well defined as soon as we have chosen an orientation of $E_0$, if we assume that $a_0$ has enough decrease at infinity to make the integral convergent. For the integral with respect to $\xi$ we  {\it define}
$$
c_n\int d\xi =1
$$
if $dx_i$ are orthonormal and oriented, where $c_n=(-1)^{n(n-1)/2}$. Except for the constant $c_n$, this is the  Berezin integral, \cite{Berezin}.
 The reason for introducing the factor $c_n$  is that we want the integral of
\be
a= a_0 dx_1\wedge d\xi_1 \wedge ...dx_n\wedge d\xi_n
\ee
to be positive if $a_0$ is positive in accordance with the complex case. If $a$ is given by (2.2) and $dx_i$ are orthonormal, we get
\be
\int_{\R^n_s} a=c_n\int_{\R^n_s} a_0 dx\wedge d\xi= \int a_0 dx\, c_n\int d\xi=\int a_0 d\lambda(x),
\ee
where $d\lambda$ is Lebesgue measure. Note  that if we change orientation, the 'superintegral' remains the same. This follows  since the integrals with respect to $x$ and $\xi$ both change sign, or directly from (2.3).

Let us briefly compare superinteration to classical integration over the complexification. The first problem with classical integration is of course that the classical integral over $E_s$ would always be divergent since $a_0$ does not depend on $\xi$. This could be overcome by replacing $E_1$ by its quotient by a lattice, so that we would replace $E_s$ by $E_0\times T^n$, where $T^n$ is a torus. The reason that does not work here is that we will later want to integrate forms of lower bidegree over linear subspaces, and these subspaces do not in general correspond to subtori, unless the subspace satisfies a rationality condition that we cannot assume to be satisfied.

We next list a few rules of computation for superintegrals. First
$$
\int_{\R^n_s} a=\int_{\R^n_s} J(a)
$$
as follows directly from the definition.  Furthermore, if $a$ is  compactly supported of bidegree $(n-1,n)$, 
$$
\int_{\R^n_s} da=0.
$$
This means that we have the usual formula for integration by parts
$$
\int_{\R^n_s} da\wedge b=(-1)^{k+1}\int_{\R^n_s} a\wedge db
$$
if at least one of the forms $a$ and $b$ is compactly supported. Similarily, using $\dhash=\pm JdJ$ and that the superintegral is $J$-invariant, we also have
$$
\int_{\R^n_s} \dhash a=0.
$$
This means that the same integration by parts formula as for $d$ also holds for $\dhash$. From this one verifies, for example, that if $\rho$ is a function and $S$ is of bidegree $(n-1,n-1)$, then
\be
\int_{\R^n_s} \rho d\dhash S=\int_{\R^n_s}d\dhash\rho\wedge S.
\ee

We can also define superintegrals over submanifolds of $\R^n$ if the form we integrate is of bidegree $(p,n)$, i. e. is of maximal degree in $d\xi$. For instance, if $D$ is a smoothly bounded domain in $\R^n$ and $a=\alpha\wedge d\xi$ is of bidegree $(n-1,n)$, then
$$
\int_{\partial D} a:= \int_{\partial D}\alpha \int_{E_1} d\xi= c_n\int_{\partial D}\alpha.
$$
 It thus follows directly that Stokes' formula holds:
$$
\int_{\partial D} a=\int_{\R^n_s} \chi_D da.
$$
'Superintegration' over a submanifold of forms not of full degree in $d\xi$ is a different matter that will be discussed later in section 4, but as a preparation for that we next discuss the case of linear subspaces. 
Let  $F$ be a linear subspace of $E=E_0$ of dimension $m$. Then its complexification is 
$F_s=F\oplus J(F)$, so $F$ defines a superspace which is a complex linear subspace of $E_s$. Restricting a superform of bidegree $(m,m)$ on $E_s$ to $F_s$, we thus have a definition of the integral of $a$ over $F_s$,
$$
\int_{F_s} a
$$
as before. 

Having defined superforms and their integrals, we now turn  to supercurrents.  The space of supercurrents of bidegree $(p,q)$, or {\it bidimension $(n-p, n-q)$},  is the dual of the space of smooth, compactly supported (in $x$ !), superforms of bidegree $(n-p,n-q)$. Here we use the classical  notion of  duals from the theory of distributions, and we say that a supercurrent is of  order zero if it is continuous for the uniform  topology on superforms. Note that a supercurrent of bidegree$(n,n)$  is a (classical) current of top degree on $E$, since  superforms  of bidegree $(0,0)$ are functions on $E$. In particular, a supercurrent of top degree and order zero is  a measure  on $E$ (not on $E_s$). As usual, given coordinates,  a supercurrent can be written 
\be
T=\sum T_{I, J} dx_I\wedge d\xi_J,
\ee
where  $T_{I,J}$ are distributions on $E_0$. Let us explain this notation a bit more. We think of a distribution as having bidegree $(0,0)$ (a generalized {\it function}), i. e. as acting on forms of top degree. Then (2.5) means that if $a$ is a test form of complementary bidegree
$$
T.\alpha=\sum T_{I,J}.( a_{K. L} dx_I\wedge d\xi_J\wedge dx_{K}\wedge d\xi_{L}).
$$
In particular, if $T_{I, J}$ are locally integrable functions, then
$$
T.\alpha=\int_{\R^n_s} T\wedge \alpha.
$$

In practice, for us, the coefficients will at worst be measures. If $\mu$ is a measure, it acts on functions, and so should be regarded as a current of bidegree $(n,n)$. We define  the corresponding object of bidegree $(0,0)$, $*\mu$, by
$$
*\mu. \alpha_0 dx_1\wedge d\xi_1 \wedge ...dx_n\wedge d\xi_n= \mu.\alpha_0.
$$
As an example of this we consider superintegration over a linear subspace $F$. Choose orthonormal coordinates so that
$$
F=\{ x_{m+1}=...x_n=0\}
$$
and write an $(m,m)$-form as
$$
a=a_0 dx_1\wedge d\xi_1...dx_m\wedge d\xi_m +...
$$
Then the restriction of $a$ to $F_s$ is just $a_0 dx_1\wedge d\xi_1...dx_m\wedge d\xi_m$, and
$$
\int_{F_s} a=\int_F a_0 *d\lambda_F
$$
as before.  Unwinding definitions, this means that 
the supercurrent defined by superintegration over $F$ is 
$$
[F]_s:= (*d\lambda_F) dx_{m+1}\wedge d\xi_{m+1} ...dx_n\wedge d\xi_m=
c_{n-m}(*d\lambda) dx_{m+1}...\wedge dx_n\wedge d\xi_{m+1}...\wedge d\xi_n.
$$
Since the standard current of integration on $F$ as a subspace of $E$ is
$$
[F]=(*d\lambda_F) dx_{m+1}...\wedge dx_n,
$$
we see that
$$
[F]_s=c_{n-m}[F]\wedge d\xi_{m+1}...\wedge d\xi_n.
$$
Note that $d\xi_i=dx_i^{\#}$ where $dx_i$ for $i=m+1, ...n$ is a frame for the conormal bundle of $F$.
In the next section we shall use this  to define superintegration over general smooth submanifolds of $E=\R^n$.

An important point to notice is that whereas the standard  current of integration is well defined without any extra structure, the supercurrent $[F]_s$ depends on the choice of scalar product. 

Just as in the complex case we now say that a supercurrent $T$  of bidimension $(m,m)$ is (weakly) positive if
$$
T. \alpha_1\wedge \alpha_1^{\#}\wedge ...\alpha_m\wedge \alpha_m^{\#}\geq 0
$$
for any choice of compactly supported $(1,0)$-forms $\alpha_j$. It is then easily verified that the supercurrent of a linear subspace is positive, and it is also immediately clear that it is $d$-closed.  

Similarily, we say that a superform $\alpha$ of  bidegree $(n-m,n-m)$ is (weakly) positive if 
\be
\alpha\wedge \alpha_1\wedge \alpha_1^{\#}\wedge ...\alpha_m\wedge \alpha_m^{\#}\geq 0
\ee
at every point. This is clearly equivalent to saying that the expression in (2.6) is positive at any point when $\alpha_j$ are constant. When 
$$
\alpha=\sum \alpha_{j k} dx_j\wedge d\xi_k
$$
is of bidegree $(1,1)$, this means that the matrix of coefficients $(\alpha_{j k})$ is positive semidefinite. At any point, such a form can be written
$$
\alpha=\sum\alpha_j\wedge\alpha_j^{\#}
$$
so we get
\begin{prop} The wedge product between a positive (symmetric) $(1,1)$-form and a positive (symmetric) form of bidegree $(p,p)$ is again positive.
\end{prop}
By approximation, the same thing holds for the product of $(1,1)$-forms with currents of bidegree $(p,p)$.

We next introduce the analog of the K\"ahler form in $\R^n$. This is by definition
$$
\beta:= \sum dx_j\wedge d\xi_j= (1/2)d\dhash |x|^2.
$$
 Then, the volume form on $F$ can be written
 $$
 d\lambda_F=[F]_s\wedge\beta^m/m!
 $$
 in good analogy with the complex case.

In the coming sections we will frequently use contraction with a 1-form in the computations. This is well defined, since we have induced scalar products on the space of all forms, and we define contraction as the dual of exterior multiplication. If $a=\sum a_j dx_j$ is a $(1,0)$-form, then $a\rfloor\beta=a^{\#}=J(a)$, whereas $a^{\#}\rfloor\beta=-a=J(a^{\#})$. Therefore,
$$
a\rfloor\beta=J(a)
$$
for any 1-form $a$. 

Similarily, contraction with a vector field 
$$
\vec{V}=\sum V_j \frac{\partial}{\partial x_j}
$$
is well defined. Lowering indices, $\vec{V}$ corresponds to the $(1,0)$-form $V=\sum V_j dx_j$, and $\vec{V}$ and $V$  act in the same way by contraction. Hence, e. g. $\vec{V}\rfloor\beta =V^{\#}$.

Finally, we will discuss how superforms transform under diffeomorphisms: If $G$ is a (local) diffeomorphism on $E=\R^n$ and $\alpha= \sum \alpha_{I, J} dx_I\wedge d\xi_J$  we define the pull back of $\alpha$ under $G$ as
\be
G^*(\alpha)=\sum G^*(\alpha_{I,J} dx_I)\wedge d\xi_J,
\ee
i. e. the pull back is the standard pull back on the $x$-part of the form and the $d\xi_j$:s are invariant. This means that we extend $G$ to a (local) diffeomorphism on $E_s$ by leaving $E_1$ fixed, and then take the usual pullback.  Then
$$
\int_{\R^n_s} G^*(\alpha)= \int_{\R^n_s} \alpha
$$
if $\alpha$ is of bidegree $(n,n)$ and $G$ is orientation preserving.  

Let now $\vec{V}$ be a vector field on $\R^n$ and let $G_t$ be its flow, or one parameter family of diffeomorphisms. The classical formula of Cartan for the Lie derivative, \cite{Warner}, says that if $\alpha$ is of bidegree $(p,0)$ then
\be
\frac{ d G_t^*(\alpha)}{dt}|_{t=0}= d (\vec{V}\rfloor\alpha) +\vec{V}\rfloor d\alpha=:L_{\vec{V}}\alpha,
\ee
where $\vec{V}\rfloor$ means contraction with the field $\vec{V}$. 
This formula holds also for forms of general bidegree. Indeed, this follows immediately from Cartan's formula on the superspace $E_s$, if we extend  $\vec{V}$  to a vector field on $E_s$ that has no component in the second factor.

\section{Relation to convex functions and tropical varieties}

It is clear from the last part of the previous section that if $\phi$ is a smooth function on $\R^n$, then $\phi$ is convex if and only if $d\dhash\phi$ is a positive superform. By approximation, it follows that  a general, possibly not smooth, function $\phi$ is convex if and only if $d\dhash\phi$ is a positive supercurrent.  Conversely, if
$$
\alpha=\sum \alpha_{j k}dx_j\wedge d\xi_k
$$
is a symmetric closed $(1,1)$ current, we can always write $\alpha=d\dhash \phi$ for some distribution $\phi$. This is because
$$ 
\frac{\partial\alpha_{j k}}{\partial x_l}=\frac{\partial\alpha_{l k}}{\partial x_j}
$$
so 
$$
\alpha_{j k}=\frac{\partial\phi_{ k}}{\partial x_j}
$$
for some $\phi_k$.
By symmetry
$$
\frac{\partial\phi_{ k}}{\partial x_j}=\frac{\partial\phi_{ j}}{\partial x_k}
$$
so $\phi_k=(\partial \phi/ \partial x_k)$ for some $\phi$.
If $\alpha$ is moreover positive, $\phi$ must be convex, so the positive symmetric $(1,1)$-currents are precisely the ones that can be written $d\dhash \phi$ for some convex $\phi$. 

If $\phi$ is a smooth convex function, we can define $(d\dhash\phi)^n/n!$ which equals
$$
c_n\det(\frac{\partial^2\phi}{\partial x_j\partial x_k})  dx\wedge d\xi.
$$
The corresponding measure on $\R^n$ is the Monge-Amp\'ere measure of $\phi$
$$
MA(\phi)= \det(\frac{\partial^2\phi}{\partial x_j\partial x_k})  dx.
$$
By the Bedford-Taylor theory ( \cite{Bedford-Taylor}) this definition makes sense for general convex functions. Indeed, Bedford and Taylor define $(dd^c\phi)^n/n!$ for general locally bounded plurisubharmonic functions, hence in particular for convex functions on $\R^n$, considered as functions on $\C^n$ that do not depend on the imaginary part of the variable.  This means that for a general convex function on $\R^n$,
$$
(d\dhash\phi)^n/n! =MA(\phi),
$$
where the Monge-Amp\'ere measure in the right hand side is taken in the sense of Alexandrov. This follows, since the two sides coincide for smooth functions and are continuous under uniform convergence.

(To be quite honest, Bedford and Taylor \cite{Bedford-Taylor} define $(dd^c\phi)^n$ as a measure on $\C^n$ and prove continuity under decreasing sequences in \cite{2Bedford-Taylor}, hence also for unform convergence when the functions are continuous. Since
$$
(d\dhash\phi)^n/n! =\pi_*(\chi(\xi)(dd^c\phi)^n/n!)),
$$
if $\chi$ is a function of $\xi$ with integral 1 and $\pi$ is the projection from $\C^n$ to $\R^n$, existence and continuity of $(d\dhash\phi)^n$ follows.)

Let us now consider convex functions of the form
$$
\phi(x)=\max_{i\in I} a^i\cdot x +b^i=\max_{i\in I} l_i(x),
$$
where $I$ is a finite set. We will call such functions, i e the maxima of a finite collection of affine functions, quasitropical polynomials, reserving the term tropical polynomials for such functions where all components of the (co)vectors $a^i$ and the numbers $b^i$ are integers, see \cite{Mikhalkin}. We may assume that all the  $a^i$ are different. Indeed, if $a^i=a^k$  and say $b^i> b^k$, then $l_i>l_k$ everywhere, and we get the same function  if we omit $l_k$. 

Let $E_j=\{\phi=l_j\}$. Then $E_j$ is defined by a finite set of linear inequalities. Assume that one of these sets, $E_k$ has empty interior. Then we can define
$$
\phi^k=\max_{i\neq k} l_i
$$
and get a new quasitropical polynomial which equals $\phi$ on a dense set. Hence, by continuity, $\phi_k=\phi$ everywhere, so we may as well omit $l_k$ in the definition of $\phi$, and can assume that all $E_j$ have non empty interior. We therefore get a decomposition of $\R^n$ as a finite union of non degenerate but possibly unbounded polyhedra. We then also have that all the $E_j$ are different, since if $a^j\cdot x+b^j=a^k\cdot x+ b^k$ on an open set, $a^j=a^k$ which we have assumed is not the case.

It is clear that
$$
\dhash\phi=\sum \chi_{E_j} (a^j)^{\#},
$$
where   $\chi_{E_j}$ is the characteristic function of the polyhedron  $E_j$.
 Since $d\chi_{E_j}=[\partial E_j]$ for some choice of orientation, we get
$$
d\dhash\phi=\sum [F_l]\wedge v_l^{\#}
$$
for some $(1,0)$ (constant) covectors $v_l$, where $F_l$ is an enumeration of the faces of the polyhedron $E_j$. These  $v_l$  must be normal to $F_l$, because $d\dhash\phi$ is symmetric (see the next section for this).  Positivity of $d\dhash\phi$  implies that they point in the same direction as the normal to $F_l$ determining the orientation. 

In this way $d\dhash\phi$ describes the {\it tropical variety} defined by the faces $F_l$, endowed with the multiplicity vectors $v_l$. (Perhaps it would be more proper to talk of quasitropical variety since the multiplicity vectors are not necessarily integral.) The fact that $d\dhash\phi$ is closed is equivalent to the {\it balancing condition} in tropical geometry: 
At a point where several faces intersect, the sum of their multiplicity vectors vanish (see \cite{Lagerberg}). Conversely, Lagerberg shows that a positive closed supercurrent of bidegree $(1,1)$ with support of dimension $n-1$ (see \cite{Lagerberg} for precise, and also more general,  statements) equals $d\dhash\phi$ for some quasitropical polynomial $\phi$.

\section{Supercurrents associated to general sumanifolds of $\R^n$}

Let $M$ be a smooth submanifold of $\R^n$ of dimension $m$. Given an orientation of $M$ we get the current of integration of $M$. Let us first assume that $M$ is a hypersurface, locally defined by an equation $\rho=0$, where $\rho$ is smooth and has nonvanishing gradient on $M$. Dividing by $|d\rho|$, we may assume that $|d\rho|=1$ on $M$, and we let $n=d\rho$; it is a unit normal form on $M$. Now it is a familiar fact that the current of integration on $M$ can be written
$$
[M]= n *dS_M
$$
where $dS_M$ is the surface measure on $M$ and the Hodge star indicates that we think of it as a current of degree zero. The choice of sign of $n$ determines the orientation of $M$. 
From this formula we see in particular that if $[M]\wedge v^{\#}$ is symmetric, then $v$ must be a multiple of $n$, hence normal to $M$ (we used this at the end of the last section, with $M=F_j$). Moreover, if $[M]\wedge v^{\#}$ is positive, $v$ is a positive multiple of $n$.

Now we define the supercurrent $[M]_s$ by
$$
[M]_s= [M]\wedge n^{\#}= n\wedge n^{\#} *dS_M.
$$
It is clearly positive and symmetric. More generally, if $M$ has codimension $p$, it is locally defined by $p$ equations $\rho_j=0$, such that $d\rho_j$ are linearly independent on $M$. Replacing $\rho_j$ by $\sum a_{j k}\rho_k=: \rho_j'$,
for a suitable matrix of functions $a_{j k}$, we may assume that $n_j:=d\rho_j$  are orthonormal on $M$. Then the currents of integration on $M$ can be written
$$
[M]=n_1\wedge...n_p *dS_M=: n*dS_M; \quad n=n_1\wedge...n_p.
$$
The supercurrent associated to $M$ is defined as 
$$
[M]_s= c_p n\wedge n^{\#} *dS_M.
$$
The definition uses the forms $n_j$ that are only locally defined, but it is easily verified that a different choice  $n_j'$ leads to the same supercurrent, since $n_j$ and $n_j'$ are related on $M$ by an orthogonal transformation.
 If $\alpha$ is a superform on the ambient space,
$$ 
\alpha\wedge [M]_s
$$
kills all the components of $\alpha$ in the directions $n_j$ and $n_j^{\#}$; we will interpret it as the restriction of $\alpha$ to $M_s$; the 'superspace' associated to $M$. Accordingly, we say that $\alpha$ vanishes on $M_s$ if (and only if) $\alpha\wedge [M]_s=0$

In general, $[M]_s$ is not closed, unless $M$ is linear. When computing $d[M]_s$ we will have use for the $(1,1)$-forms
$$
F:=d n^{\#}, \quad F_j=d n_j^{\#}.
$$
These forms depend on the choice of defining function(s), but transform in a natural way when we change defining functions. When the codimension $p=1$, the restriction  of $F$ to $M$ is the second fundamental form of $M$. In higher codimension the $F_j$ restricted to $M$ are the components of the vector valued second fundamental form $\sum F_j\otimes \vec{n}_j$ with values in the normal bundle of $M$. 

Again, we start with the case $p=1$. Then 
$$
d[M]_s= d([M]\wedge n^{\#})=-F\wedge [M],
$$
since $[M]$ is $d$-closed.
This expresses $d[M]_s$ in terms of $[M]$ but we want to write it in terms of $[M]_s$. Therefore we introduce the operator
$$
\F= F\otimes n^{\#}\rfloor
$$
which acts on a superform or supercurrent by first contracting with $n^{\#}$ and then wedging with $F$. Notice that $\F$ is an antiderivation. Then
\be
d[M]_s=\F [M]_s.
\ee
In a similar way, when $p>1$ we let
$$
\F=\sum_j F_j\otimes n_j^{\#}\lfloor,
$$
and still get $d[M]_s=\F[M]_s$. In the same way we get that
\be
\dhash [M]_s = -\F^{\#} [M]_s,
\ee
where 
$$
\F^{\#}=\sum_j F_j      \otimes n_j\lfloor.
$$
Because of the following lemma, we can also write $\F=\sum_j n_j^{\#}\lfloor\otimes F_j$, i. e. we can first wedge with $F_j$ and then contract. 
\begin{lma}
$ n_j\rfloor F_j=0 =n_j^{\#}\rfloor F_j$ on $M_s$, i. e. when wedged with $[M]_s$.
\end{lma}
\begin{proof} We prove this when $p=1$. This is only  to simplify the index notation; in the proof $n_j$ denotes the components of $n$, i. e. the partial derivatives of $\rho$ with respect to $x_j$. Then
$$
n\rfloor F=\sum n_j \rho_{j k}d\xi_k= 1/2\sum \frac{ \partial n_j^2}{\partial x_k}d\xi_k=1/2 \dhash |n|^2.
$$
Since the norm of $n$ is constant on $M$, this vanishes on $M_s$.
\end{proof}

\section{Minimal submanifolds}
We start with the following computational proposition. 
\begin{prop}
If $M$ is a smooth $m$-dimensional submanifold of $\R^n$
$$
d\left([M]_s\wedge \beta^p\right)= \sum_j \n_jhash\rfloor\left(F_j\wedge [M]_s\wedge \beta^p\right).
$$
\end{prop}
\begin{proof}
  We have, since $d\beta=0$, 
  $$
  d\left([M]_s\wedge \beta^p\right)=\sum F_j\wedge n_j^{\#}\rfloor [M]_s\wedge\beta^p.
  $$
  Since, by Lemma 4.1, $n_j^{\#}\rfloor (F_j)\wedge[M]_s=0$, this equals
  $$
  \sum n_j^{\#}\rfloor(F_j\wedge[M]_s)\wedge\beta^p.
  $$
  Finally, $n_j^{\#}\rfloor\beta=-n_j$, which vanishes when wedged with $[M]_s$.
  This gives the formula in the proposition.
\end{proof}
If $F=\sum F_{i j}dx_i\wedge d\xi_j$ is a $(1,1)$-form,
 
$$F\wedge\beta^{n-1}/(n-1)!= tr(F)\beta^n/n!,
$$ 
 where $tr(F)=\sum F_{i i}$ is the trace of $F$. This implies that
$$
F\wedge[M]_s\wedge\beta^{m-1}/(m-1)!= tr'(F) [M]_s\wedge\beta^m/m!,
$$
where $tr'(F)$ is the trace of $F$:s restriction to $M_s$, since wedging with $[M]_s$ kills all the components of $F$ and $\beta$ in the normal directions. The traces $tr'(F_j)=:H_j$ are the coefficients of the {\it mean curvature vector}
$$
\vec{H}:=\sum H_j\vec{n_j}.
$$
(There seem to be different conventions as to the sign of the mean curvature vector. We follow here the convention in \cite{2Colding-Min}, so that the mean curvature of a sphere points outwards.)

Applying Proposition 5.1 with $p=m-1$ we get
\be
d([M]\wedge\beta^{m-1}/(m-1)!)=\sum H_j n_j^{\#}\rfloor([M]_s\wedge\beta^m/m!).
\ee
Since $n_j$ are linearily independent, this vanishes exactly when  $\vec{H}$ vanishes, i. e. when $M$ is a {\it minimal manifold}, so we have proved

\begin{cor}
$[M]_s\wedge\beta^{m-1}$ is a closed current if and only if $M$ is a minimal submanifold.
\end{cor}
We are therefore led to the following 
\medskip

  \noindent{\bf Definition:}   {\it A positive symmetric supercurrent $T$ of bidimension $(m,m)$ is minimal if 
  $$d T\wedge \beta^{m-1}=0$$}.

    \qed

  \medskip
  
 It is clear that this property is conserved by regularisation, e.g. by convolution with an approximate identity.  For a general positive supercurrent of bidimension $(m,m)$, if it happens that there is a vector field $\vec{V}=\sum V_j \partial/\partial x_j$ such that
 $$
 dT\wedge \beta^{m-1}/(m-1)!= \sum V_j d\xi_j\rfloor(T\wedge\beta^m/m!),
 $$
 we say that $\vec{V}$ is a mean curvature vector for $T$. As we have seen, this is the case when $T=[M]_s$ is associated to a smooth manifold, but it certainly also holds when $T$ is strictly positive and smooth. In both these cases,  $\vec{V}$ is uniquely determined,  so we may speak of {\it the} mean curvature vector. In section 8 we shall see that
 $$
 d\dhash (T\wedge\beta^{m-1}/(m-1)!)
 $$
 can be computed in terms of the flow of the mean curvature vector field when it exists.

On a minimal supercurrent we can also define a Dirichlet form:
\be
D(u,u):=\int_T du\wedge \dhash u\wedge\beta^{m-1}/(m-1)!=T.du\wedge \dhash u\wedge\beta^{m-1}/(m-1)!,
\ee
if $u$ is sufficiently smooth on $\R^n$. When $T=[M]_s$ is the supercurrent associated to an $m$-dimensional manifold, this is precisely the standard Dirichlet form
$$
\int_M |d_M u|^2 dS_M,
$$
where $d_M u$ is the differential of $u$ restricted to $M$, and $|d_Mu|$ is the norm  induced by the Euclidean metric on $\R^n$. 

If $u$ has compact support, this equals
$$
\int_T ud\dhash u\wedge\beta^{m-1}/(m-1)!.
$$
It is therefore natural to define the Laplacian of $u$ by
$$
\Delta_T u =d\dhash u\wedge T\wedge \beta^{m-1}/(m-1)!.
$$
With this definition, $\Delta_T u$ is a measure on $\R^n$. In some cases, e.g. when $T$ is the supercurrent of a manifold, or a smooth strictly positive form, we can write
$$
\Delta_T u=\tilde\Delta_T u T\wedge\beta^m/m!,
$$
where $\tilde\Delta_T$ is a scalar valued Laplacian. In any case, we say that 
 $u$ is {\it harmonic (subharmonic) on $T$} if $\Delta_T u=0$ (or $\Delta_T u\geq 0$). Notice that, just as in the case of K\"ahler metrics, the Laplacian on a minimal supercurrent has no first order terms, so  linear functions are harmonic. For minimal manifolds this is a well known property, cf \cite{Colding-Min}.  
\section{  Volume estimates for minimal submanifolds}
To prove volume estimates for minimal manifolds (or supercurrents), we will now follow the method of Lelong to prove such estimates in the complex setting. 
 This requires
one little twist since the minimal supercurrent $T$ (e.g. $T=[M]_s$) is not closed itself; it is only $T\wedge\beta^{m-1}$ that is closed. This is taken care of by the following lemma.
\begin{lma}
Let for $\delta\geq 0$, $|x|_\delta:=({|x|}^2+\delta)^{1/2}$,  and let for  $p>0$
$$ 
E_{p,\delta}:=\frac{-1}{p} |x|_\delta^{-p},
$$
and for $p=0$ 
$$
E_{0,\delta}=\log |x|_\delta.
$$
Then, if $p$ is an integer,  
$$
d\dhash E_{p,\delta}\wedge\beta^{p+1}=(d\dhash |x|_\delta)^{p+2}.
$$
\end{lma}
\begin{proof}
This is a direct computation and we will do it for $p>0$, the case $p=0$ being similar but simpler. First,
$$
\dhash E_{p,\delta}=|x|_\delta^{-(p+2)}(1/2)\dhash |x|^2
$$
and 
\be
d\dhash E_{p,\delta}=|x|_\delta^{-(p+2)}\left( \beta- (p+2)\frac{\gamma}{{|x|_\delta}^2}\right),
\ee
with $\gamma =d|x|^2\wedge \dhash |x|^2/4$.  On the other hand
$$
d\dhash |x|_\delta =|x|_\delta^{-1}(\beta-\frac{\gamma}{{|x|_\delta}^2}).
$$
Expanding by the binomial theorem we get
\be
(d\dhash |x|_\delta)^{p+2}=|x|_\delta^{-(p+2)}\left(\beta^{p+2}-(p+2)\frac{\gamma}{{|x|_\delta}^2}\wedge \beta^{p+1}\right).
\ee
The lemma follows from (6.1) and (6.2).
\end{proof}
The reason we consider $E_{p, \delta}$ instead of $E_{p,0}$ is just that we want our functions to be smooth across zero. The main conclusion we draw from the lemma is that
$$
d\dhash E_{m-2,\delta}\wedge\beta^{m-1}\wedge T
$$
is positive if $T$ is positive of bidegree $(m,m)$. This follows from Proposition (2.1) since $|x|_\delta$ is convex. We also remark that it follows from the proof of the proposition that
\be
d\dhash E_{p,\delta}\wedge \beta^{m-1}\wedge T\geq 0
\ee
if $0\leq p\leq m-2$.

Let $T$ be a minimal current of bidimension $(m,m)$, defined in a neighbourhood of the origin. Its mass in a ball of radius $r$ centered at the origin is
$$
\sigma(r):=\int_{|x|<r} T\wedge\beta^m/m!.
$$
In the computations below we first assume that $T$ is smooth. Then, writing $S=T\wedge \beta^{m-1}/m!$
$$
\sigma(r)=
$$
\be
\int_{|x|=r} (1/2)\dhash |x|^2\wedge S= \int_{|x|=r} |x|_\delta^{m} \dhash E_{m-2,\delta}\wedge S= (r^2+\delta)^{m/2}\int_{|x|<r}(d\dhash |x|_\delta)^m \wedge T/m!,
\ee
since $dS=0$. Since the integrand in the right hand side is nonnegative it follows that
$$
(r^2+\delta)^{-(m/2)}\sigma(r)
$$
is (weakly) increasing. This holds for any $\delta>0$, so $ r^{-m}\sigma(r)$ is also increasing. By approximation with smooth forms, this holds also for general minimal currents, so we have proved the following generalization of the {\it monotonicity theorem} for minimal manifolds (see \cite{Colding-Min}).
\begin{thm} Let $T$ be a minimal (super)current of bidimension $(m,m)$ defined in a neighbourhood of the origin. Then
  $$
  r^{-m}\int_{|x|<r} T\wedge\beta^m/m!
  $$
  is nondecreasing.
\end{thm}
When $T$ is the supercurrent of a minimal manifold, this says that the area of the  manifold inside a ball of radius $r$, divided by $r^{-m}$ is nondecreasing.
In analogy with the case of minimal manifolds we call
$$
\gamma_T(0):=\lim_{r\to 0} r^{-m}\int_{|x|<r} T\wedge\beta^m/m!,
  $$
the density of $T$ at the origin. When $T$ is the supercurrent of a (smooth)
minimal manifold it equals $\omega_m$, the volume of $m$-dimensional unit ball. (The corresponding limit for  closed positive $(m,m)$-currents in complex analysis is the {\it Lelong number} of the current.)

Let us now look again at the formula
$$
\int_{|x|<r} T\wedge\beta^m/m!=(r^2+\delta)^{m/2}\int_{|x|<r}T\wedge(d\dhash |x|_\delta)^m/m!.
$$
We proved this for $T$ smooth, but by approximation it holds for general minimal supercurrents of bidimension $(m,m)$. It means in particular that the integral in the right hand side is bounded as $\delta\to 0$, so there is a subsequence of $\delta$:s such that
$$
T\wedge (d\dhash |x|_\delta)^m/m!
$$
converges weakly to a measure, $\mu$. Outside the origin, where $|x|$ is smooth, this measure must be
$$
T\wedge(d\dhash |x|)^m/m!.
$$
We also see that $\mu$ must have a point mass at the origin of size $\gamma_T(0)$, since the mass of $\mu$ in any  ball centered at the origin with small radius $r$ is
$$
r^{-m}\int_{|x|<r} T\wedge\beta^m/m!.
$$
Thus, every subsequence has the same limit
$$
\mu=\chi_{x\neq0}T\wedge(d\dhash |x|)^m/m! +\gamma_T(0)\delta_0.
$$
By Lemma 4.3 we also have
$$
\lim_{\delta\to 0} T\wedge d\dhash E_{m-2,\delta}\wedge \beta^{m-1}/m!=\mu,
$$
which can be interpreted as saying that the 'Laplacian' of $E_{m-2,0}$ on $T$ (the trace of $d\dhash E_{m-2, 0}$) equals a point mass at the origin of size $\gamma_T$, plus a nonnegative contribution outside the origin. The contribution outside the origin vanishes when $T$ is the supercurrent of an $m$-dimensional plane through the origin. This reflects the fact that $E_{m-2,0}$ is a fundamental solution of the Laplacian then; the crucial observation is that $E_{m-2,0}$ is always 'subharmonic on $T$'.

We shall next generalize the proof of Theorem 6.2 to a general domain, $D$. Then $|x|$ is not constant on the boundary of $D$, so instead we write  $|x|^m =w(x)$ on the boundary , where $w$ is a positive smooth function on the closure of $D$ to be chosen later.
Let $T$ be a minimal supercurrent (which we tacitly take as smooth at first) in a neighbourhood of $\bar D$. Then the mass of $T$ in $D$ is
$$
\int_D T\wedge\beta^{m}/m!= \int_{\partial D} \dhash |x|^2/2\wedge S/m!=\int_{\partial D} w \dhash E_{m-2}\wedge S/m!.
$$
(Here we skip the part of the argument where we approximate $|x|$ by $|x|_\delta$ and we write $E_{m-2}$ for $E_{m-2,0}$.) By Stokes' theorem this equals
$$
\int_D w d\dhash E_{m-2}\wedge S/m! +\int_D d w\wedge \dhash E_{m-2}\wedge S/m!=: I + II
$$
since $S$ is closed.
By what we have just seen, $I\geq w(0)\gamma_T(0)$. To see when $II$ is positive we compute (for $m>2$)
\be
0=(m-2)^{-1}\int_{\partial D} \left( \frac{1}{|x|^{m-2}} -\frac{1}{w^{1-2/m}}\right) \dhash w\wedge S=
\ee
$$
\frac{1}{m-2}(\int_D\left( \frac{1}{|x|^{m-2}} -\frac{1}{w^{1-2/m}}\right)d\dhash w\wedge S+m^{-1}\int_D\frac{1}{w^{2-2/m}}dw\wedge\dhash w\wedge S -\int_D dE_{m-2}\wedge\dhash w\wedge S. 
$$ 
Since
$$
d\phi\wedge\dhash\psi\wedge S=d\psi\wedge\dhash \phi\wedge S
$$
when $S$ is symmetric we find that
$$
II=\frac{1}{m-2}\int_D\left( \frac{1}{|x|^{m-2}} -\frac{1}{w^{1-2/m}}\right)d\dhash w\wedge S/m!+m^{-1}\int_D\frac{1}{w^{2-2/m}}dw\wedge\dhash w\wedge S/m!.
$$
It follows that $II\geq 0$ if $w\geq |x|^m$ in $D$ and $w$ is convex.
\begin{thm} Let $D$ be a smoothly bounded domain in $\R^n$.  Let $a$ be any point in $D$ and let $w$ be a convex function on $\bar D$ such that $w=|x-a|^m$ on the boundary of $D$. If $T$ is a minimal supercurrent in $D$ of bidimension $(m,m)$, then its mass satisfies
  $$
  \int_D T\wedge\beta^m/m!\geq w(a)\gamma_T(a)
  $$
  (where $\gamma_T(a)$ is the density of $T$ at $a$).
\end{thm}
(We shall see in the next section that the convexity assumption on $w$ can be relaxed considerably.)

In the proof we may of course  assume that $a=0$. 
For $m>2$ the theorem follows immediately from the argument above, since we can always assume that $w\geq |x|^m$, replacing it if necessary by $\max(w, |x|^m)$. The case $m=2$ is similar; we just have to replace the boundary integral in (6.5) by
$$
\int_{\partial D} \log \frac{w}{|x|^2} \dhash w\wedge S.
$$
As a special case we get a recent result of Brendle and Hung  (\cite{Brendle-Hung}), which generalizes an older estimate of Alexander and Osserman for the case $m=2$ (\cite{Alexander-Osserman}).
\begin{cor} Let $M$ be a minimal manifold in the unit ball which contains the point $a$. Then the volume of $M$ satisfies
  $$
  |M|\geq \omega_m (1-|a|^2)^{m/2}.
  $$
\end{cor}
\begin{proof} We have
  $$
  |x-a|^2= 1+|a|^2 -2a\cdot x=:v
  $$
  on the boundary of the ball. Since $v$ is convex (in fact linear), $w:=v^{m/2}$ is also convex and the corollary follows directly from the theorem.
\end{proof}
\section{ Removable singularities}
The techniques of the previous section can also be used to give a variant of the El Mir-Skoda theorem on extension of positive closed currents (\cite{El Mir}, \cite{Skoda}) in the setting of minimal manifolds or minimal currents. We first define a locally integrable function $u$ to be {\it m-subharmonic} if
$$
d\dhash u\wedge \beta^{m-1}\wedge \alpha
$$
is a positive current for any (weakly) positive form $\alpha$ of bidimension $(m,m)$. (In the terminology of complex analysis this amounts to saying that $d\dhash u\wedge \beta^{m-1}$ be {\it strongly positive}, see \cite{Lelong}.) Then (6.3)  says that $E_{p,\delta}$ is $m$-subharmonic if $0\leq p\leq m-2$, and taking limits when $\delta\to 0$ the same thing holds for $E_p$. Therefore, any potential
\be
u(x):=\int \frac{-1}{|x-y|^{m-2}}d\mu(y)
\ee
for $m>2$ and
$$
u(x):=\int\log |x-y|d\mu(y)
$$
when $m=2$ is also $m$-subharmonic, if $\mu$ is a positive measure.

Kernels like $E_{m-2}$ on $\R^n$ are called Riesz kernels, and have a well developed potential theory, following the lines of the more classical potential theory for the Newtonian kernel $E_{n-2}$, see \cite{Landkof}. Thus, we have a notion of capacity $C_{m-2}$ associated to $E_{m-2}$ and any set of sigma-finite $(m-2)$-dimensional Hausdorff measure has capacity zero. 

We also say that a  set $F$ is $m$-polar is there is an $m$-subharmonic function which is equal to $-\infty$ on $F$ (and maybe elsewhere as well).
By  \cite{Landkof}, any compact set $K$ with $C_{m-2}(K)=0$ is $m$-polar, and moreover there is a measure $\mu$ supported on $K$ whose potential equals $-\infty$ precisely on $K$. Therefore we get
\begin{prop} Any compact  set $K$ of $\sigma$-finite $(m-2)$-dimensional Hausdorff measure is  $m$-polar and there is a potential $u$ of a measure supported on $K$ which equals $-\infty$ on $K$.
\end{prop}
To illustrate this, notice that it is   immediate that a discrete set of points is 2-polar, and that a submanifold of dimension $(m-2)$ is $m$-polar in general. Indeed, it suffices to take $\mu$ equal to surface measure.
\begin{thm} Let $T$ be a minimal supercurrent of bidimension $(m,m)$, defined in $B\setminus K$, where $B$ is a ball and $K$  and is compact in $\R^n$ with sigma-finite $(m-2)$-dimensional Hausdorff measure. Assume that $T$ has finite mass
  $$
  \int_{B\setminus K} T\wedge \beta^m <\infty.
  $$
  Then the trivial extension of $T$, $\tilde T:=\chi_{B\setminus K} T$ , is a  minimal supercurrent in $B$.
\end{thm}
\begin{proof}
  We follow almost verbatim the proof of Proposition 11.1 in \cite{B-Sibony}. 
  Let $u$ be a potential as in (4.7) with $\mu$ supported on $K$, which equals $-\infty$ on $K$. Thus, $u$ is smooth outside $K$ and $m$-subharmonic. Let $\chi(t)$ be an increasing continuous convex function, defined when $t\leq 0$, with $\chi(0)=1$ and $\chi(t)=0$ for $t\leq -1$. Put
  $$
  u_k=\chi(u/k)
  $$
  for $k=1,2, ...$. Then $u_k$ are smooth, $m$-subharmonic, $0\leq u_k < 1$, and $u_k$ tend to 1 uniformly on compacts outside $K$. 

  Let $\theta$ be a smooth function with compact support in $B$. Then, by integration by parts,
  $$
  \int \theta T\wedge\beta^{m-1}\wedge d\dhash u_k^2\leq C
  $$
  where $C$ is a fixed constant independent of $k$. This implies that 
  $$
  \int\theta T\wedge\beta^{m-1}\wedge du_k\wedge \dhash u_k \leq C,
  $$
  since $u_k$ is $m$-subharmonic.

  Let $p(x)$ be a smooth function on $\R$, such that $p(x)=1$ if $x>1/2$ and $p(x)=0$ if $x<1/3$, and put $\chi_k=p(u_k)$. Then $\chi_k$ tends to $\chi_{B\setminus K}$ and we get if $\psi$ is a testform
$$
  \int d\left(\chi_{B\setminus K} T\wedge\beta^{m-1}\right)\wedge \psi =\lim_{k\to\infty} \int d\left(\chi_k T\wedge\beta^{m-1}\right)\wedge \psi =\lim_{k\to\infty}\int p'(u_k) d u_k\wedge T\wedge\beta^{m-1}\wedge \psi. 
$$
 By our estimate  on $d u_k$ this limit is zero.  Therefore $\chi_{B\setminus K}T\wedge\beta^{m-1}$ is $d$-closed, so we are done.   
\end{proof}

One consequence of this is that Corollary 6.4 also holds for minimal manifolds with singularities. More precisely,
by the theorem, the extension $\tilde T$ has a density also at points in $K$. It is easy to see, by the monotonicity theorem,  that
$$
\gamma_{\tilde T}(a)\geq \limsup_{b\to a}\gamma_{\tilde T}(b).
$$ 
If $T$ is associated  with  a minimal manifold of dimension $m$ that has singularities at the set $E$  the density is $\omega_m$ (the volume of an $m$-dimensional unit ball) at all the regular points. Therefore it is at least $\omega_m$ at the singular points. 

(A natural question in this context seems to be if minimal manifolds can be characterized within the class of minimal supercurrents by conditions on the density. For instance, if $T$ is minimal and the density is constant on the support of $T$, is $T$ then $c[M]_s$ for a smooth minimal manifold $M$ ?)

We also note that the assumption that $w$ be convex in Theorem 4.5 can be relaxed to $m$-subharmonic, since we just need that
$$
d\dhash w\wedge S=d\dhash w\wedge T\wedge \beta^{m-1}
$$
be nonnegative.

\section{Variation of volume and the mean curvature flow}

We have seen in formula (5.1), that if $M$ is a smooth manifold of dimension $m$, then
\be
d([M]_s\wedge\beta^{m-1}/(m-1)!)=\sum H_j n_j^{\#}\rfloor ([M]_s\wedge\beta^m/m!),
\ee
where $H_j$ are the components of the mean curvature vector.
In this section we will use the notation $\beta_m=\beta^m/m!$,
$$
S=[M]_s\wedge\beta_{m-1},
$$
and
$$
\sigma=[M]_s\wedge\beta_m.
$$
Then formula (8.1) says that
\be
dS=H^{\#}\rfloor \sigma.
\ee
Since $S$ is $J$-invariant and of even degree it follows after a small computation that
\be
\dhash S= JdS=-H\rfloor\sigma.
\ee

Our main objective in this section is to give a formula for the derivative of the volume form
$$
\sigma_M:=[M]_s\wedge \beta^m/m!
$$
when $M$ moves under {\it the mean curvature flow},  see \cite{2Colding-Min}, \cite{Smoczyk}. As in these references,  we say that a family $M_t$ of smooth submanifolds of $\R^n$  moves by the  mean curvature flow if there is a vector field $\vec{H}$ on the ambient space such that $F_t(M_0)=M_t$, where $F_t$ is the flow of $-\vec{H}$,  which restricts to the mean curvature vector field of $M_t$ on each $M_t$.

  Again we point out that we use here the same conventions about the sign of the mean curvature vector as in \cite{2Colding-Min}, which seems to be the opposite of the one in \cite{Smoczyk}. To make matters worse, the condition
$$
F_t(M_0)= M_t,
$$
means in terms of pullbacks of  currents that
$$
F_{-t}^*([M_0])=[M_t],
$$
so since we will work with pullbacks, we are dealing with the flow $G_t$ of the vector field $+\vec{H}$ after all. Recalling Cartan's formula for the Lie derivative (cf section 2) we get
\be
\frac{d}{dt}|_{t=0}[M_t]=L_{\vec{H}} [M]=d \vec{H}\rfloor [M]=d H\rfloor [M],
\ee
since $[M]$ is closed.

We first note that the next proposition follows directly from (8.3) and Cartan's formula.

\begin{prop}
  $$
  \frac{d F_{-t}^*(\sigma)}{dt}|_{t=0}=-d\dhash S.
  $$
\end{prop}

We now let $\sigma_t=[M_t]_s\wedge\beta_m$; the volume form of $M_t$. (Note that this is not the same thing as $F_{-t}^*(\sigma)$ which has total mass independent of $t$, whereas the mass of $M_t$ changes.) The next theorem is the main result of this section. 
 
  \begin{thm}  

 Let $M_t$ be a family of smooth $m$-dimensional  submanifolds of $\R^n$, moving under the mean curvature flow, with $M_0=M$. Let
  $$
  \sigma_t=\sigma_{M_t}=[M_t]_s\wedge \beta_m
  $$
  be their volume forms. Then
  \be
  \frac{d\sigma_t}{dt}|_{t=0}=-|\vec{H}|^2\sigma -d\dhash S.
  \ee
\end{thm}

  To prove the theorem we need a small preparation. First recall that locally on $M$ we can find an orthonormal basis for the (co)normal bundle, $(1,0)$-forms $n_j$, $j=1,...p$. We now claim that we can find such forms $n_j(t)$ that depend smoothly on $t$,  form an orthonormal basis for the (co)normal bundle of $M_t$  on $M_t$ and moreover are orthonormal everywhere in a small neighbourhood. Indeed, such forms satisfying the first two conditions are constructed as in section 4 (when we were dealing with a fixed manifold). We then just apply the Gram-Schmidt procedure (which changes nothing on $M_t$) to make them orthonormal everywhere in a small neighbourhood. 

  This means that in particular   the $p$-forms $n(t):=n_1(t)\wedge... n_p(t)$ have norm one, so that $\dot n$ is perpendicular to $n$. Recalling that
  $$
  [M_t]_s= c_p[M_t]\wedge n^{\#}(t)
  $$
  and using (8.4) we then get
  $$
  \frac{d}{dt}[M_t]_s= c_p(d(H\rfloor[M_t])\wedge n^{\#}(t)+[M_t]\wedge \dot n^{\#}(t)).
$$
  Wedging this with $\beta_m$ the last term disappears since $\dot n$ is perpendicular to $n$. Thus
  $$
  \frac{d\sigma_t}{dt}|_{t=0}= c_p d(H\rfloor[M_t])\wedge n^{\#}(t)\wedge\beta_m=
   c_p d((H\rfloor[M_t])\wedge n^{\#}(t)\wedge\beta_m)+(-1)^{p-1}c_p((H\rfloor[M])\wedge dn^{\#}\wedge\beta_m.
   $$
   The first term on the right hand side equals
   $$
   c_p d((H\rfloor[M_t])\wedge n^{\#}(t)\wedge\beta_m)=d(H\rfloor\sigma)=
   -d\dhash S,
$$
   by (8.3). For the second term we use that
   $$
   dn^{\#}=\sum F_j\wedge n_j^{\#}\rfloor n^{\#}= \sum n_j^{\#}\rfloor F_j\wedge n^{\#}
   $$
   by Lemma 4.1.
   Using this and that $n_j^{\#}\rfloor\beta=n_j$ we get that
   $$
   dn^{\#}\wedge\beta_m=-\sum n_j\wedge F_j\wedge n^{\#}\wedge \beta_{m-1}.
   $$
Now we plug this into the second term and get   
$$
(-1)^{p-1}c_p((H\rfloor[M])\wedge dn^{\#}\wedge\beta_m=
-c_p(\sum n_j\wedge(H\rfloor[M])\wedge n^{\#}\wedge F_j\wedge \beta_{m-1}
   $$
which equals
$$
-\sum H_j[M]_s\wedge F_j\wedge \beta_{m-1}=-|\vec{H}|^2\sigma.
$$
This completes the proof of the theorem. \qed

Integrating (8.5) against a smooth function $\rho$, and using (2.4),  we obtain

\begin{cor}
  Let $M_t$ be a family of compact manifolds of dimension $m$, moving under the mean curvature flow. Then, if $\rho$ is a sufficiently smooth function
  $$
  (d/dt)|_{t=0}\int \rho d\sigma_t= -\int \rho |\vec{H}|^2 d\sigma -\int d\dhash\rho\wedge S.
    $$
    As a consequence, if $\rho$ is convex (or, more generally, $m$-subharmonic), then
    $$
    \int \rho d\sigma_t
    $$
    is decreasing. 
    \end{cor}
It follows from the last part that if $M$ lies in a convex open set, then $M_t$ will remain there, since we may choose $\rho$ to be close to infinity outside the convex set.

Formula (8.5) reflects the fact that the mean curvature flow is described by a (non linear) parabolic equation. It says that the volume forms then flow by a (linear) parabolic equation. It seems to be a natural question if there is a corresponding equation for the supercurrents $[M_t]_s$ themselves, i. e. without taking traces as we have done here. Since the minimal surface equation  becomes linear in the supercurrents formalism, this is perhaps not complete unrealistic.

\section{General submanifolds and their Riemannian geometry}
In this section we shall describe  the Levi-Civita connection and Riemannian curvature of a submanifold $M$ of $\R^n$ in terms of the 'superstructure'.

Let $M$ be a smooth submanifold of $\R^n$ and let $[M]_s$ be its associated supercurrent. With the notation from  section 3 we have
$$
d[M]_s= \F[M]_s, \quad \dhash [M]_s= -\F^{\#} [M]_s.
$$
Motivated by these formulas we introduce the operators
$$
D:=d-\F, \quad D^{\#}:\dhash +\F^{\#}
$$
and get
$$
D[M]_s=D^{\#} [M]_s=0.
$$
Notice that $D$ and $D^{\#}$ are antiderivations, so they  satisfy the same rules of computation as $d$ and $\dhash$. If $a$ is a superform on the ambient space we get
$$
D (a\wedge[M]_s)= (Da)\wedge [M]_s, \quad D^{\#}(a\wedge [M]_s)=(D^{\#}a) \wedge [M]_s.
$$
In particular, if $a$ vanishes on $[M]_s$, i.e. $a\wedge [M]_s=0$, then $Da$ and $D^{\#} a$ also vanish on $[M]_s$,  which means that $D$ and $D^{\#}$ are well  defined as operators on forms on $[M]_s$.

Since $\F$ involves contraction with $n_j^{\#}$, it vanishes on any form of bidegree $(p,0)$. Hence $D=d$ on such forms and we have just retrieved the fact that the exterior derivative is well defined on submanifolds. On the other hand, $\F^{\#}$ involves contraction with $n_j$ and does not vanish on forms of bidegree $(p,0)$, so $\dhash$ and $D^{\#}$ are different on $(p,0)$ forms. In fact, $D^{\#} a$ can be thought of as the Levi-Civita connection acting on $a$. Indeed, $D^{\#} a$ is a form of bidegree $(p,1)$, so it can be seen as a 1-form in $\xi$ with values in the space of $(p,0)$-forms in $x$. If $V=\sum V_j \partial/\partial x_j$ is a tangent vector to $M$, the derivative of $a$ in the direction $V$ is
$$
V^{\#}\rfloor D^{\#} a,
$$
where $V^{\#}=\sum V_j\partial/\partial \xi_j$. It is not hard to verify that this is the Levi-Civita connection. 

To compute its curvature we need to apply $D^{\#}$ twice. This is done in the following proposition which is just a way of writing Gau\ss's formula for the curvature  in supernotation (see also Weyl, \cite{Weyl}). 
\begin{prop} The curvature tensor $R$ on $M$ is given by the symmetric $(2,2)$ form
  $$
  R=(1/2)\sum F_j\wedge F_j,
  $$
  in the sense that, if $a$ is of bidegree $(1,0)$
  $$
  (D^{\#})^2 a= a\rfloor (1/2) F_j\wedge F_j.
  $$
\end{prop}
If $a$ is a form of bidegree $(p,q)$ on $M_s$ we can always extend it to a $(p,q)$-form on the ambient space. It is convenient to choose a special extension. We say that the extension is canonical if $n_j\rfloor a =n_j^{\#}\rfloor a=0$ for $j=1,...p$. It is obvious that canonical extensions exist locally since $n_j=d\rho_j$ are linearly independent and can be completed to a basis for the $(1,0)$ forms. We then just take any extension, write it in terms of the new basis, and throw away the terms that contain some $n_j$ or $n_j^{\#}$. The resulting form coincides with $a$ on $M_s$ in the sense that its wedge product with $[M]_s$ equals $a\wedge [M]_s$.

If $a$ is  a canonical extension of a form on $M_s$, then $\F a=\F^{\#} a=0$, so
$$
D^{\#} a= \dhash a.
$$
Thus
$$
(D^{\#})^2 a =\F^{\#} \dhash a=\sum F_j\wedge n_j\rfloor \dhash a.
$$
To continue, we introduce the notation that if $F=\sum F_{i j}dx_i\wedge d\xi_j$ is a $(1,1)$ form (on $\R^n_s$) and $a$ is $(p,q)$ form, then 
$$
F\cup a =-\sum F_{i j} d\xi_j\wedge dx_i\rfloor a.
$$
Notice that in the particular case when $a$ has bidegree $(1,0)$
$$
F\cup a= -a\rfloor F.
$$

\begin{lma} If $F= d\theta^{\#}$ (where $\theta$ is $(1,0)$), then we have the commutator formula
  $$
  [\dhash, \theta\rfloor]a:= \dhash\theta\rfloor a +\theta\rfloor \dhash a= -F\cup a
  $$
  where $a$ is any superform (canonical or not). 
\end{lma}
\begin{proof} First we note that the commutator is a scalar operator, so that 
  $$
  [\dhash, \theta\rfloor] fa=f [\dhash, \theta\rfloor]a
  $$
  if $f$ is a function. Therefore we may assume that $a=dx_I\wedge d\xi_J$.
  Then $\dhash a=0$ and
  $$
  \dhash\theta\rfloor a=\sum\frac{\partial\theta_i}{\partial x_j}d\xi_j\wedge dx_i\rfloor a,
  $$
  which proves the formula. 
  
\end{proof}
The lemma gives that
$$
(D^{\#})^2 a= -\sum F_j\wedge F_j\cup a,
$$
for $a$ of any bidegree. Combined with the remark immediately before the lemma, this gives that
$$
(D^{\#})^2 a=a\rfloor (1/2)\sum F_j\wedge F_j
$$
if $a$ is $(1,0)$, which concludes the proof of the proposition.

From Proposition 8.1 we see that  integrals like
$$
\int_{[M]_s} \chi (\sum F_j\wedge F_j)^q\wedge \beta^{m-2q},
$$
where $\chi$ is a function on $M$ are intrinsic, i. e. they do not depend on the embedding of $M$ in $\R^n$, since the Riemann curvature (and the metric) are intrinsic. (This is of course not true for similar integrals containing arbitrary combinations of $F_j$.)   In the next section we shall give an illustration of this.
\section{Weyl's tube formula.}

We will now use the formalism of the previuos section  to give a quick proof of Weyl's tube formula (\cite{2Weyl}, \cite{Gray}). The proof is not really different from the original proof, but the formalism helps to organise the formulas.

Let $M$ be a compact submanifold of $\R^n$.  We will consider $T_r(M)$, the tube around $M$ of width $r$, which when $M$ is without boundary is defined as 
$$
T_r(M)=\{x\in\R^n; d(x,M)<r\}.
$$
When $r$ is sufficiently small, $T_r(M)$ is by the tubular neighbourhood theorem diffeomorphic to a neighbourhood of the zero section of the normal bundle $N(M)$ of $M$, 
$$
\Delta(M,r):=\{v\in N(M); |v|<r\}.
$$
A point $v$ in the normal bundle is a vector, normal to to $T_p(M)$ where $p=\pi(v)$, $\pi$ being the projection from $N(M)$ to $M$, and the diffeomorphism is
$$
G: \Delta(M,r) \to \R^n, \quad G(v)= \pi(v)+ v.
$$
When $M$ is a compact manifold {\it with} boundary, we {\it define} $T_r(M)$ to be the image of $\Delta(M,r)$ under this map. Weyl's tube formula is the following theorem.

 \begin{thm} Let $M$ be a compact  $m$-dimensional submanifold of $\R^n$ and let $T_r(M)$ be the tube of width $r$ around $M$. Then, for $r$ small
    $$
    |T_r(M)|= \sum_{2q\leq m} c_{2q} r^{2q+p}\int_{[M]_s} R^q\wedge \beta^{m-2q}
    $$
    for some constants $c_{2q}$ that can be explicitly calculated, where  $p=n-m$ is the codimension of $M$.
  \end{thm}

Thus the theorem says first that the volume of the tubes of width $r$ is a polynomial in $r$ for $r$ small, and moreover and most remarkably that the coefficients of the polynomial are intrinsic. Thus, if we have two isometric embeddings of the Riemannian manifold $M$ into $\R^n$, the tube volumes are the same.

For the proof, we have first
 $$
  |T_r(M)|=\int_{[T_r(M)]_s} (\sum dx_k\wedge d\xi_k)^n/n!.
 $$
 It is enough to prove the theorem locally; i. e. we may assume that the normal forms $n_j=d\rho_j$ of section 4 are defined on all of $M$. Then the normal bundle is trivial and isometric to $M\times B_r(0)$, where $B_r(0)$ is the ball in $\R^n$ of radius $r$ and center 0, via the orthonormal frame $\vec{n_j}$. Hence the diffeomorphism $G$ described above can be written
 $$
  G:M\times B_r(0) \to T_r(M),
  $$
  with
  $$
  G(y,t)= y+\sum t_j\vec{n_j}.
  $$
  Pulling back by $G$ as in section 2 we get
  $$
  |T_r(M)|=\int_{M\times B_r(0)}(\sum dy_k\wedge d\xi_k +\sum t_j d n_j^{\#}+\sum dt_j\wedge n_j^{\#})^n/n!=
    $$
    $$
   = \int_{M\times B_r(0)}(\sum dy_k\wedge d\xi_k +\sum t_j F_j+\sum dt_j\wedge n_j^{\#})^n/n!
    $$
    Although $M\times B(0,r)$ is not strictly speaking a domain in $\R^n$, the integral here should be interpreted as an ordinary integral with respect to $(y,t)$ and the Berezin integral with respect to $\xi$.
      We expand the integrand by the trinomial theorem ( since all three forms within the parenthesis are of degree 2 they commute), and get
    
 \be
  |T_r(M)|=\sum_l c_{l,m}\int_{|t|<r} dt\int_{[M]_s} \beta^l\wedge (\sum t_j F_j)^{m-l}.
  \ee   
  Here $\beta=\sum dy_k\wedge d\xi_k$, and we are using 
  $$
  [M]\wedge(\sum dt_j\wedge n_j^{\#})^{n-m}= c[M]_s dt_1..\wedge dt_k
  $$
  for some constant $c$. What remains is therefore to compute the form valued integral
  $$
  \int_{|t|<r} (\sum t_jF_j)^{m-l}dt.
  $$
  We first compute the integral
  \be
   \int_{|t|<r}(\sum t_ja_j)^{m-l} dt,
   \ee
  where $a=(a_1, ...a_p)$ lies in $\R^p$. It clearly vanishes when $m-l$ is odd, and when $m-l=2q$ it equals a constant times 
  $$
  |a|^{2q} r^{2q+p}.
  $$
  This follows since the integral is rotational invariant and homogenous of degree $2q$ in $a$, and homogenous of order $2q+p$ in $r$.
  Hence, up to a constant,  the integral in (10.2) equals
  \be
  (\sum a_j^2)^qr^{2q+p}
  \ee
From a different point view, we could have expanded the integrand in (10.2) and obtained a linear combination of monomials in $a$  of degree $m-l$. The fact that the resulting homogenous polynomial in $a$ is given by (10.3) is therefore an algebraic identity. 
  This must also hold when $a_j$ are not real numbers but lie in any commutative algebra over the reals, like in our case when $a_j=F_j$ are two-forms. 
  Hence
  $$
   \int_{|t|<r}(\sum t_jF_j)^{m-l} dt= c_q r^{2q+p}(\sum F_j^2)^q=c_q'r^{2q+p} R^q.
   $$
   Inserting this into (10.1), the theorem follows.

\end{document}